\documentclass[11pt]{amsart}
\usepackage{a4wide}
\usepackage{amsmath}
\usepackage{amssymb} 
\usepackage{amsfonts}
\usepackage{bbm}
\usepackage{mathrsfs}
\usepackage{mathdots}
\usepackage{multicol}
\usepackage{color}
\usepackage{tikz}
\usepackage{tikz-cd}
\numberwithin{equation}{section}

\theoremstyle{plain}
\newtheorem{theorem}{Theorem}[section]

\newtheorem{lemma}[theorem]{Lemma}

\theoremstyle{definition}
\newtheorem{definition}[theorem]{Definition}

\theoremstyle{remark}

\newcommand{\leg}[2]{\left( \frac{#1}{#2} \right)}
\newcommand{\kzxz}[4]{\left(\begin{smallmatrix} #1 & #2 \\ #3 & #4\end{smallmatrix}\right) }

\newcommand{\calG}{\mathcal{G}}

\newcommand{\calK}{\mathcal{K}}
\newcommand{\calL}{\mathcal{L}}

\newcommand{\calU}{\mathcal{U}}

\newcommand{\calZ}{\mathcal{Z}}
\newcommand{\A}{{\mathbb A}}

\renewcommand{\H}{\mathbb{H}}
\newcommand{\N}{\mathbb{N}}

\newcommand{\Z}{\mathbb{Z}}
\newcommand{\Q}{\mathbb{Q}}
\newcommand{\R}{\mathbb{R}}
\newcommand{\C}{\mathbb{C}}
\newcommand{\frake}{\mathfrak e}
\newcommand{\bs}{\backslash}
\newcommand{\SO}{\operatorname{SO}}

\newcommand{\SL}{\operatorname{SL}}
\newcommand{\GL}{\operatorname{GL}}

\newcommand{\sig}{\operatorname{sig}}

\newcommand{\Ker}{\operatorname{Kern}}

\newcommand{\id}{\operatorname{id}}

\newcommand{\diag}{\operatorname{diag}}
\newcommand{\odd}{\operatorname{oddity}}

\newcommand{\affiliation}{%
\parbox{1.05 \textwidth}{\centering
  Fakult\"at f\"ur Informatik und Mathematik\\ Ostbayerische Technische Hochschule Regensburg\\
Galgenbergstraße 32, 93053 Regensburg, Germany}}

\newcommand{\emailaddress}{%
\parbox{1.05 \textwidth}{\centering
  \texttt{oliver.stein@oth-regensburg.de}}}

\begin{document}

\title[vector-valued automorphic forms]{Construction of vector-valued adelic automorphic forms}
\author{Oliver Stein}
\maketitle

\affiliation

\vspace{0.3cm}

\emailaddress

%
%

\section{Introduction}
Vector-valued modular forms for the Weil representation play a central role in the theory of automorphic forms. They arise naturally in the theory of theta functions, Borcherds products, and arithmetic geometry, and have found numerous applications in the study of orthogonal Shimura varieties, special cycles, and generating series (see for example \cite{Br1} or \cite{YY}). Their close relation to scalar-valued modular forms has been investigated by several authors. In particular, explicit lifting maps between scalar-valued and vector-valued modular forms have been constructed and their mapping properties studied in considerable detail (see e.g.  \cite{BB}, \cite{Sch1}, \cite{Zh1} or \cite{Zh2}).

From the adelic point of view, scalar-valued modular forms are naturally realized as automorphic forms on $\GL_2(\A)$, thereby placing them into the framework of automorphic representations. An adelic realization of vector-valued modular forms for the Weil representation was recently established in \cite{St}. Since the lifting constructions developed here rely heavily on this adelic realization, we recall in sufficient detail its principal constructions, emphasizing those aspects needed for the subsequent lifting theory. The main purpose of the present note is to show that the classical lifting maps between scalar-valued and vector-valued modular forms admit natural adelic analogues. More precisely, we construct lifting maps between the corresponding spaces of automorphic forms and prove that they are compatible with the classical lifting maps under adelization.

The contents of this paper are as follows: We recall the adelic realization of scalar-valued and vector-valued modular forms (cf. Sections \ref{sec:scalar_valued_automorphic_forms} and \ref{sec:vec_val_automorphic_forms}). A technical ingredient of the paper is the observation that scalar-valued automorphic forms on $\GL_2(\A)$ may be replaced by their restrictions to a subgroup $G(\A)$ without loss of information (cf. Theorem \ref{rem:relation_automorphic_forms}). This places the scalar-valued and vector-valued theories on the same adelic group and thereby provides the common framework needed for the lifting constructions.  We then construct in Section \ref{sec:lifting_automorphic_forms} adelic lifting maps between scalar-valued and vector-valued automorphic forms and prove that they are compatible with the classical lifting maps under adelization. Thus, the correspondence between scalar-valued and vector-valued modular forms extends consistently to the adelic setting.

Besides providing a conceptual reformulation of the classical theory, the adelic viewpoint opens the possibility of studying these lifting maps by means of the theory of automorphic representations. It is reasonable to expect that the constructions considered here induce corresponding relations between the automorphic representations attached to scalar-valued and vector-valued automorphic forms. This suggests several directions for future research. Among them are the study of the associated automorphic representations, the investigation of the mapping properties of the adelic lifting maps in analogy with the classical results, and the compatibility of these lifts with Hecke operators.

This paper is based on material presented at the annual {\it Symposium on Automorphic Representations, Automorphic Forms, and L-functions} held at the {\it Research Institute for Mathematical Sciences (RIMS)}, Kyoto, in January 2026. I am very  grateful to Professor Gunji for the invitation to present my work and for organizing this excellent conference, and I thank the organizing committee for their efforts in making it a stimulating meeting. The investigations leading to the present paper were  motivated by a question raised during the discussion following my talk.

\section{Groups}
We begin by introducing the groups and notation that will be used throughout the paper.
As usual, we let $e(z)$, $z\in \C$, be the abbreviation for $e^{2\pi i z}$. For any prime number $p$, by $\Q_p$ we mean the field of $p$-adic numbers and by $\Z_p$ its ring of $p$-adic integers; $|\cdot|_p$ is the $p$-adic absolute value and $\nu_p(\cdot)$ the $p$-adic valuation of $\Q_p$. We write $\A$ for the adele ring of $\Q$ and $\A^\times$ for the idele group. By $\A_f$, we mean the set of finite adeles, and by $\widehat{\Z} = \prod_{p < \infty}\Z_p$.
For any commutative ring $R$ with $1$, we write $1_2$ for the unit matrix in $M_{2,2}(R)$, 
\begin{equation}\label{def:p_adic_subgroups}
  \begin{split}
 &  G(R) = \{ M\in \GL_2(R)\; |\; \det(M) \in (R^\times)^2\} \text{ and } G(\A) =\sideset{}{'} \prod_{p \le \infty} G(\Q_p),\\  
    &\calG(R) = \{(M, r)\in \GL_2(R)\times R^\times\; |\; \det(M) = r^2\} \text{ and } \calG(\A) = \sideset{}{'}\prod_{p \le \infty} \calG(\Q_p),\\
    & Z(R) = \{ \kzxz{r}{0}{0}{r}\; |\; r\in R^\times\} \text{ and } \calZ(R) = \{(z, r)\in Z(R)\times R^\times\; |\; \det(z) = r^2\},\\
    & \overline{G}(R) = Z(R)\bs G(R),\; \overline{\calG}(R)= \calZ(R)\bs \calG(R) \text{ and } \overline{\GL_2}(R) \text{ accordingly},\\
     & \calK_p = \calG(\Z_p) \text{ and } \calK = \prod_{p < \infty} \calK_p.\\
    & \text{ For } M\in \N, K_0(M)_p = \begin{cases}
      \left\{\kzxz{a}{b}{c}{d}\in \GL_2(\Z_p)\; |\: c\equiv 0\bmod{M}\right\}, & p\mid M,\\
      \GL_2(\Z_p), & p\nmid M
    \end{cases} \\
&    \text{ and } K_0(M) = \prod_{p < \infty} K_0(M)_p, \\
    & \calK_0(M)_p= \begin{cases}
      \left\{\left(\kzxz{a}{b}{c}{d}, r\right)\in \GL_2(\Z_p)\times \Z_p^\times\; |\; \det\kzxz{a}{b}{c}{d} = r^2 \text{ and } c\equiv 0\bmod{M}\right\}, & p\mid M\\
      \calK_p, & p\nmid M,
    \end{cases}\\
&    \text{ and } \calK_0(M) = \prod_{p < \infty}\calK_0(M)_p. 
    \end{split}
\end{equation}
Here $\prod'$ is the usual symbol for the restricted product. For our purposes, it is sufficient to  identify $\calG(\Q_\infty)$ with $\GL_2(\R)^+$ by the section $g \mapsto (g,\sqrt{g})$.
For later reference, we note that
\begin{equation}
  \begin{split}
    & (g, r)\mapsto (gr^{-1}, r) \text{ defines an isomorphism of }  \calG(\A) \text{ and } \SL_2(\A)\times \A^\times \text{ and }\\
    & \SL_2(\A)\times \A^\times \rightarrow G(\A), \quad (h, r)\mapsto hr \text{ is a homomorphism with kernel } \{(u1_2, u^{-1})\; |\; u^2 = 1\}. 
  \end{split}
\end{equation}
Therefore, $G(\A)\cong (\SL_2(\A)\times \A^\times)/\{(u1_2, u^{-1})\; |\; u^2 = 1\}$.
Under these two isomorphisms the center of these two groups corresponds to
\begin{equation*}
  \calZ(\A) \cong Z(\SL_2(\A))\times \A^\times \text{ and } Z(\A) \cong (Z(\SL_2(\A))\times \A^\times)/\{(u1_2, u^{-1})\; |\; u^2 = 1\}.
\end{equation*}
Quotienting by the respective centers therefore yields
\begin{equation}\label{eq:quotient_center}
  \calZ(\A)\bs \calG(\A) \cong Z(\SL_2(\A))\bs \SL_2(\A) \text{ and } Z(\A)\bs G(\A) \cong Z(\SL_2(\A))\bs \SL_2(\A)  
\end{equation}
and so the quotients $\overline{\calG}(\A)$ and $\overline{G}(\A)$ are isomorphic. The same holds for $\overline{\calG}(\Q)$ and $\overline{G}(\Q)$. 

Since any element $(g,r)$ in $\SL_2(\A)\times \A^\times$ is unipotent if and only if $g$ is unipotent in $\SL_2(\A)$ and $r$ is unipotent in $\A^\times$, the maximal unipotent group for this group is $N(\A)\times \{1\}$, which in turn means that the maximal unipotent group $N_\calG(\A)$ of $\calG(\A)$ is equal to
\begin{equation}\label{eq:unipotent_group}
\left\{\left(\kzxz{1}{t}{0}{1}, 1\right) \; |\; t\in \A\right\} \cong N(\A) = \left\{\kzxz{1}{t}{0}{1} \; |\: t\in \A\right\}.
\end{equation}

Finally, we use the symbol $\Gamma$ for the modular group $\SL_2(\Z)$, $\Gamma_0(M)$ and $\Gamma(M)$ for the usual congruence subgroup and principal congruence subgroup, respectively. 


The groups $\calG$ are more suitable for our purposes.  For $\calG(\A)$ a strong approximation theorem analogous to the one for the classical groups holds:
\begin{theorem}\label{thm:strong_approx_pairs}
  Let $\nu: \calG(\A_f)\rightarrow \A_f^\times,\quad (g,r)\mapsto r$ and $\calU = \prod_{p<\infty}\calU_p$ be any open compact subgroup of $\calG(\widehat{\Z})$ with the property that $\nu(\calU) = \widehat{\Z}^\times$. Then
  \begin{equation}\label{eq:strong_approx_pairs}
    \begin{split}
      & \calG(\A_f) = \calG(\Q)\cdot \calU \text{ and }\\
      & \calG(\A) = \calG(\Q)(\GL_2(\R)^+ \times \calU).
      \end{split}
    \end{equation}
  \begin{proof}
 The idea is to reduce the statement to the corresponding strong approximation theorem for $\SL_2$ via the homomorphism $\nu$. Indeed, under the homeomorphism $(g, 1)\mapsto g$ the group  $\Ker(\nu)$ identifies with  $\SL_2(\A_f)$, in particular $\calG(\widehat{\Z})\cap \Ker(\nu)\cong \SL_2(\widehat{\Z})$. 

 As $\calU$ is a subgroup of $\calG(\widehat{\Z})$, we obtain the subgroup  $\calU_1=\calU \cap \Ker(\nu)$ of $\Ker(\nu)$. 
Since $\calU$ is open, $\calU\cap \Ker(\nu)$ is open 
in $\Ker(\nu)$. Moreover,  $\Ker(\nu)$ is closed in $\calG(\widehat{\Z})$ and $\calU$ is compact. Hence, $\calU_1$ is an open compact subgroup of $\calU$ and its image, denoted by $U_1$, under the above identification is a subgroup in $\SL_2(\widehat{\Z})$ with the same properties. 

Now let $(g,r)\in \calG(\A_f)$. In view of the decomposition $\A_f=\Q^\times\widehat{\Z}^\times =  \Q^\times \nu(\calU)$, we may write  $\nu(g) = q \nu(u)$ for some $u=(g_u,r_u)\in \calU$. This implies
\[
(g_1,1) = (q1_2,q)^{-1}(g,r)(g_u,r_u)^{-1} \in \Ker(\nu)\cong\SL_2(\A_f).
\]

By strong approximation for $\SL_2$ with respect to $U_1$, we find  $\gamma \in SL_2(\Q)$ and $u_1\in U_1$ such that $(g_1,1) = (\gamma,1)(u_1,1)$ and therefore
\[
(g,r) = (q\gamma,q)(u_1g_u,r_u)\in \calG(\Q)\cdot \calU.
\]
This decomposition extends immediately to $\calG(\A)$: given $g = (g_\infty, g_f)$ with $g_f = \widetilde{\gamma} u_f$ with $\widetilde{\gamma}=(\gamma,r) \in \calG(\Q)$ and $u_f\in \calU$, the identity 
\[
g =  \widetilde{\gamma}(\gamma^{-1}g_\infty, u_f),
\]
follows immediately. Here we notice that $\gamma^{-1}g_\infty\in \GL_2(\R)^+$ as $\gamma\in G(\Q)$ has a positive determinant. 
  \end{proof}
\end{theorem}

Observe that  $\nu(\calK_0(M)) = \widehat{\Z}^\times$ holds since for every $u\in \widehat{\Z}^\times$,
\[
\left(\diag(1,u^2), u\right) \in \calK_0(M).
\]
Consequently, Theorem \ref{thm:strong_approx_pairs} applies to $\calK_0(M)$.

\section{scalar-valued adelic automorphic forms}\label{sec:scalar_valued_automorphic_forms}
The aim of this section is to recall the adelic realization of scalar-valued automorphic forms. Together with the vector-valued theory recalled in the next section, it provides the framework for the lifting constructions established in Section \ref{sec:lifting_automorphic_forms}. We also fix some notation used throughout this paper. For details we refer to \cite{Mi} and \cite{KL}. Let $\kappa\in \Z$. The group $\GL_2^+(\Q)$ acts on the space of functions $f:\H\rightarrow \C$ via
\begin{equation}\label{eq:Petersson_Slash}
  (f\mid_\kappa A)(\tau) = \det(A)^{\kappa/2}j(A,\tau)^{-\kappa}f(A\tau), \tau\in \H,
\end{equation}
where, as usual, $j(A,\tau) = c\tau+d$ for $A=\kzxz{a}{b}{c}{d}\in \GL_2(\Q)^+$.
Let $M\in \N$,  $\chi$ a Dirichlet character and $M_\kappa(\Gamma_0(M),\chi)$ the space of modular forms for $\Gamma_0(M)$ with character $\chi$. As usual, we write $S_\kappa(\Gamma_0(M),\chi)$ for the subspace of cusp forms.  Any such form satisfies \newline $f\mid_\kappa \gamma = \chi(\gamma)f$ for all $\gamma\in \Gamma_0(M)$, where $\chi(\gamma) = \chi(d)$ for $\gamma=\kzxz{a}{b}{c}{d}\in \Gamma_0(M)$. 

By $\xi$, we denote the Hecke character induced by $\chi$ and by $A_\kappa(K_0(M),\xi)$, we mean the space of cuspidal automorphic forms on $\GL_2(\A)$ of weight $\kappa\in \Z$, level $K_0(M)$ and character $\xi$ associated to $S_\kappa(\Gamma_0(M),\chi)$. Here we briefly recall the definition of $A_\kappa(K_0(M),\xi)$, which is taken from \cite{KL}, Prop. 12.5.

The space $A_\kappa(K_0(M),\xi)$ consists of all functions $\phi: \GL_2(\A)\rightarrow \C$ satisfying
\begin{enumerate}
\item[i)]
  $\phi\in L^2(\GL_2(\Q)\bs \GL_2(\A),\xi)$ and $\displaystyle \int_{N(\Q)\bs N(\A)}\phi(ng)dn = 0$ for almost every $g\in \GL_2(\A)$, where $N$ is the usual unipotent group in $\GL_2(\A)$
\item[ii)]
  $\phi(gk) = \xi(k)\phi(g)$ for all $k\in K_0(M)$  and $g\in \GL_2(\A)$
\item[iii)]
  $\displaystyle \phi(g\kzxz{\cos(\theta)}{\sin(\theta)}{-\sin(\theta)}{\cos(\theta)}) = e^{i\kappa \theta}\phi(g)$ for all $\theta\in [0,2\pi)$ and all $g\in \GL_2(\A)$
\item[iv)]
  The function $\phi$ as a function of $\GL_2(\R)^+$ alone satisfies the differential equation $L \phi = 0$. Here $L$ is the differential operator given by
  \[
  L = e^{-2i\theta}\left(-2iy\frac{\partial}{\partial x} + 2y\frac{\partial}{\partial y} + i\frac{\partial}{\partial \theta} \right),
  \]
  with respect to the coordinates of the Iwasawa decomposition of $\GL_2(\R)^+$.
  \end{enumerate}
Here the $L^2$-space is defined by
\begin{equation}\label{eq:L2_scalar_valued}
L^2(\GL_2(\Q)\setminus \GL_2(\A), \xi) = \left\{\phi\rightarrow \C\; \left|
    \begin{array}{ll}
      \text{i)} &  \phi \text{ is measurable } \\
      \text{ii)} & \phi(zg) = \xi(z)\phi(g) \text{ for all }\\
      & z\in Z(\A)\\
      \text{iii)} & \int_{\overline{\GL_2}(\Q)\setminus \overline{\GL_2}(\A)}|\phi(g)|^2 dg <\infty
      \end{array}
    \right.\right\},
    \end{equation}
     where $d g$ is any right $\GL_2(\A)$-invariant measure on $\overline{\GL_2}(\Q)\setminus \overline{\GL_2}(\A)$.
    
 In terms of this decomposition we can define the usual adelization map $\mathscr{A}_\xi$ (see e.g. \cite{KL}, Sect. 12.2) for scalar-valued modular forms for $\Gamma_0(M)$ with character $\chi$. To this end, let $g = g_\Q(g_\infty \times k)\in \GL_2(\A)$ and put 
\begin{equation}\label{eq:adelization_map_scalar}
f \mapsto  \mathscr{A}_{\xi}(f) = F_f,\text{ where } g\mapsto F_f(g_\Q(g_\infty \times k)) = \xi(k)j(g_\infty, i)^{-\kappa}f(g_\infty i).
\end{equation}

It is easily confirmed that the restriction of any $\phi\in A_\kappa(K_0(M),\xi)$ to the subgroup $G(\A)$ yields an adelic automorphic form on the group $G(\A)$ of the same weight with respect to the level $K_0(M)\cap G(\A)$ and character $\xi_{|_{\calG(\A)}}$. For a fixed character $\xi$ and level $K_0(M)$ there is close relation between these spaces of automorphic forms. In fact, we have 

\begin{theorem}\label{rem:relation_automorphic_forms}
  The map
  \begin{equation}\label{eq:restriction}
  \phi \mapsto \phi_{|_{\calG(\A)}}
  \end{equation}
  defines an isomorphism of the spaces $A_\kappa(K_0(M), \xi)$ and $A_\kappa(\calK_0(M)\cap G(\A),\xi_{|_{G(\A)}})$.
  \end{theorem}
\begin{proof}
  The proof relies on the identity
  \begin{equation}\label{eq:decompostion_GL2}
    \GL_2(\A) = G(\A)\cdot \GL_2(\Q)\cdot K_0(M) = \GL_2(\Q)\cdot G(\A)\cdot K_0(M),
  \end{equation}
  where the latter equation makes use of the fact that $G(\A)$ is a normal subgroup of $\GL_2(\A)$. The first equation is a consequence of 
  \begin{equation}\label{eq:decomposition_A_times}
  \A^\times = (\A^\times)^2\cdot \Q^\times\cdot \widehat{\Z}^\times,
  \end{equation}
which translates to \eqref{eq:decompostion_GL2} by applying it to the determinant of a matrix $g\in \GL_2(\A)$.
  The identity \eqref{eq:decomposition_A_times} can be verified by a direct computation similar to the proof of the strong approximation theorem of $\A^\times$ (see e.g. \cite{KL}, Prop. 5.10).

  The injectivity of \eqref{eq:restriction} is immediate from decomposition \eqref{eq:decompostion_GL2} and the identity    
  \begin{equation}\label{eq:phi_GL_2}
    \phi(g) = \xi(k)\phi(g_G) \text{ for } g= g_\Q g_G k,
  \end{equation}
 which expresses every value of $\phi$ in terms of its restriction to $G(\A)$.  Consequently, if $\phi_{|_{G(\A)}} = 0$, it follows from \eqref{eq:phi_GL_2} that  $\phi$ is zero on the whole group $\GL_2(\A)$.

  Conversely, any automorphic form $\phi_0\in A_\kappa(K_0(M)\cap G(\A),\xi_{|_{G(\A)}})$ admits a  unique lift to an automorphic form $\phi \in A_\kappa(K_0(M), \xi)$: For $g =  g_\Q g_G k$ we put
  \begin{equation}\label{eq:lift_GL2_G}
    \phi(g) = \xi(k)\phi_0(g_G).
  \end{equation}
  We need to check that this definition is independent of the decomposition of $g$. To this end, let
\begin{equation}\label{eq:equivalences}
  g= g_\Q g_G k = g_\Q' g_G' k' \Longleftrightarrow k'k^{-1} = g_G'^{-1}g_\Q'^{-1}g_\Q g_G \Longleftrightarrow g_\Q'^{-1}g_\Q =  g_G'k' k^{-1} g_G^{-1}. 
\end{equation}
In terms of these two representations of $g$, it remains to prove that
\begin{equation}\label{eq:obstruction}
  \begin{split}
  \phi(g_\Q g_G k) = \phi(g_\Q' g_0' k') & \Longleftrightarrow \xi(k)\phi_0(g_G) = \xi(k')\phi_0(g_G') \\
  & \Longleftrightarrow \phi_0(g_G) = \xi(k'k^{-1})\phi_0(g_G'). 
  \end{split}
\end{equation}
But the latter identity is valid since $g_\Q g_G k = g_\Q' g_G' k'$ is equivalent to
\begin{equation}\label{eq:rel_g_g_prime}
g_G = g_\Q^{-1}g_\Q' g_G' k'k^{-1}.
\end{equation}
Further, because $g_G\in G(\A)$, the same holds for the right-hand side of \eqref{eq:rel_g_g_prime} such that by \eqref{eq:lift_GL2_G}
\begin{equation}\label{eq:values_phi}
\phi(g_G) = \phi(g_\Q^{-1}g_\Q' g_G' k'k^{-1}) \Longleftrightarrow \phi_0(g_G) = \phi_0(g_\Q^{-1}g_\Q' g_G' k'k^{-1}).   
\end{equation}
Provided we can  confirm that $g_\Q^{-1}g_\Q'\in G(\Q)$ and $k'k^{-1}\in K_0(M)\cap G(\A)$,  the desired identity in \eqref{eq:obstruction} follows from \eqref{eq:values_phi}.

Thanks to \eqref{eq:equivalences}, $g_\Q'^{-1}g_\Q\in \GL_2(\Q)\cap G(\A)K_0(M)G(\A) = \GL_2(\Q)\cap G(\A)K_0(M)$, where the latter identity once again makes use of the fact that $G(\A)$ is normal in $\GL_2(\A)$.
Now taking determinants of any matrix $g_Gk$ in $\GL_2(\Q)\cap G(\A)K_0(M)$ yields an element in $\Q^\times\cap (\A^\times)^2\widehat{\Z}^\times$, by the definition of these groups. Since any $x\in \Q^\times\cap (\A^\times)^2\widehat{\Z}^\times$ has even $p$-adic valuation (the valuation of any $u\in \widehat{\Z}^\times$ is zero and does not contribute), we have that $x\in (\Q^\times)^2$ and therefore $\Q^\times\cap (\A^\times)^2\widehat{\Z}^\times = (\Q^\times)^2$. Consequently,
\begin{equation}
  \begin{split}
    \GL_2(\Q)\cap G(\A)K_0(M) &=  \left\{\gamma\in \GL_2(\Q)\; |\; \det(\gamma) \in \Q^\times \cap (\A^\times)^2\widehat{\Z}^\times\right\}\\
    & = \left\{\gamma\in \GL_2(\Q)\; |\; \det(\gamma) \in (\Q^\times)^2\right\}, 
\end{split}
\end{equation}
which is precisely the group $G(\Q)$. Taking this fact and the identity
\[
k'k^{-1} = g_G'^{-1}g_\Q'^{-1}g_\Q g_G
\]
from \eqref{eq:equivalences} into account, we can conclude that $k'k^{-1}\in K_0(M)\cap G(\A)G(\Q)G(\A) = K_0(M)\cap G(\A)$ since $G(\Q)\cap G(\A)$.

The remaining defining properties of an automorphic form follow immediately from the corresponding properties of $\varphi_0$. 
\end{proof}



\section{vector-valued adelic automorphic forms}\label{sec:vec_val_automorphic_forms}
The purpose of this section is twofold. First, we recall the theory of vector-valued automorphic forms in sufficient detail for the reader's convenience. Second, we establish the framework needed to formulate and prove the lifting results of Section \ref{sec:lifting_automorphic_forms}.
More specifically, we define  vector-valued adelic automorphic forms associated  with vector-valued modular forms for the ``finite'' Weil representation $\rho_D$ and describe the corresponding adelization map. To this end, we first introduce the Weil representation $\omega = \omega_\infty\otimes \omega_f$ of $\SL_2(\A)\times H(\A)$ and study a subrepresentation closely related to $\rho_D$. 

\subsection{The Weil representation}
Let $L$ be a lattice of rank $m$  equipped with a non-degenerate $\Z$-valued bilinear form $(\cdot,\cdot)$ of type $(b^+,b^-)$ such that the associated quadratic form
\[
q(x):=\frac{1}{2}(x,x),\quad x\in L,
\]
takes values in $\Z$. We assume that $m$ is even  and denote  its signature $b^+-b^-$ by $\sig(L)$. Note that $\sig(L)$ is also even. We stick with these assumptions on $L$ for the rest of this paper unless stated otherwise. We further  define $V(\Q)=L\otimes \Q$ and let $H=O(V)$ be the orthogonal group over $\Q$ attached to $(V(\Q),(\cdot,\cdot))$. We denote by
\[
L':=\{x\in V(\Q) \; |\; (x,y)\in\Z\quad \text{ for all } \; y\in L\}
\]
the dual lattice of $L$.
 Since $L\subset L'$, the elementary divisor theorem implies that $L'/L$ is a finite group. We denote  this group by $D$. The modulo 1 reduction of both the bilinear form $(\cdot, \cdot)$ and the associated quadratic form, defines a $\Q/\Z$-valued bilinear form $(\cdot,\cdot)$ with corresponding $\Q/\Z$-valued quadratic form on $D$. We call $D$ equipped with $(\cdot,\cdot)$ a discriminant form or a discriminant group.   By $\C[D]$, we mean the group algebra of $D$ and denote by $\{\frake_\lambda\}_{\lambda\in D}$ its standard basis. Further,
\begin{equation}\label{eq:scalar_product_group_ring}
    \left\langle \sum_{\lambda\in D}a_\lambda\frake_\lambda,\sum_{\lambda\in D}b_\lambda\frake_\lambda \right\rangle =\sum_{\lambda\in D}a_\lambda \overline{b_\lambda}.
  \end{equation}
is the standard scalar product on $\C[D]$.
We denote by $N$ the level of the lattice $L$. It is the smallest positive integer such that $Nq(\lambda)\in \Z$ for all $\lambda\in L'$. Throughout this note, by $N$ we always mean the level of $L$ and assume that $N$ is {\it odd}. As $N$ and $|D|$ share the same prime factors, $|D|$ is also odd.
 The ``finite'' Weil representation associated to $D$ (or $L$) is a unitary representation of $\Gamma$ on the group ring $\C[D]$ with respect to $\langle\cdot,\cdot\rangle$. One important property of $\rho_D$ is that it is trivial on the principal congruence subgroup $\Gamma(N)$. Thus, it factors through
 \[
 \Gamma/\Gamma(N) \cong \SL_2(\Z/N\Z).
 \]
 For further details see e.g. \cite{Sch2} or \cite{BS}. Note that $D$ can be decomposed into $p$-groups $\displaystyle D = \bigoplus_{p\mid |D|}D_p$. 
On the level of $\C[D]$ and $\rho_D$  this decomposition translates to 
$\displaystyle \C[D] = \bigotimes_{p||D|}\C[D_p]$
and
$\displaystyle \rho_D = \bigotimes_{p||D|}\rho_{D_p}$,
respectively, where $\rho_{D_p}$ means the Weil representation on $\C[D_p]$ attached to the quadratic module $D_p$.
Later in this note, we will compute $\omega_f$ on $\calK_0(N)$. To this end, we need the following lemma, which can be found in \cite{Sch2}.

\begin{lemma}[\cite{Sch2}, Prop. 4.7]\label{lem:weil_rep_gamma_0_N}
  Let $D$ be a discriminant form of even signature and level $N$ and $\gamma=\kzxz{a}{b}{c}{d}\in \Gamma_0(N)$. Moreover, let $r$ an integer. By $g_r(D)$ we denote the Gauss sum
  \begin{equation}\label{eq:gauss_sum_d}
    g_r(D)=\sum_{\lambda\in D}e(rq(\lambda))
  \end{equation}
  and we set $g(D) = g_1(D)$. Then $\gamma$ acts in the Weil representation of $D$ as
  \begin{equation}\label{eq:weil_rep_gamma_0_N}
    \rho_D(\gamma)\frake_\lambda = \frac{g(D)}{g_a(D)}e(-bdq(\lambda))\frake_{d\lambda},
  \end{equation}
\end{lemma}
\begin{proof}
  By Prop. 4.7 in \cite{Sch2}, we find
  \begin{equation}\label{eq:rho_D_Gamma_0_N}
    \rho_D(\gamma)\frake_\lambda = \leg{a}{|D|}e((a-1)\odd(D)/8)e(-bdq(\lambda))\frake_{d\lambda}.
  \end{equation}
 Lemma 2.1 in \cite{St} expresses
\begin{equation}\label{eq:oddity}
  \leg{a}{|D|}e((a-1)\odd(D)/8)
\end{equation}
  as the quotient of Gauss sums $\frac{g(D)}{g_a(D)}$. 
\end{proof}
In addition, if $|D|$ is odd, the quantity in \eqref{eq:oddity} simplifies to a quadratic character
\begin{equation}\label{eq:chi_D}
  \chi_D(a) = \leg{a}{|D|}. 
\end{equation}

To define Hecke operators on vector-valued modular forms for $\rho_D$, Bruinier and the author constructed in \cite{BS} an extension of $\rho_D$ from $\SL_2(\Z/N\Z)$ to the group
\[
Q(N)=\{(M,r)\in \GL_2(\Z/N\Z)\times (\Z/N\Z)^\times \;|\; \det(M)\equiv r^ 2\bmod{N}\}.
\]

 The relation between $\rho_D$ and a subrepresentation of the adelic Weil representation $\omega$ is crucial for the extension of $\omega_f$ to $\calK$, which in turn is  necessary  for  the definition of vector-valued adelic automorphic forms. Here we consider the Weil representation of $\SL_2(\A)\times H(\A)$ acting on the space $S(V(\A))$ of Schwartz-Bruhat functions, where $V(\A) = V\otimes \A$. It depends on the selection of an additive character $\psi$  of $\A/\Q$. Here we pick the  complex conjugate  of the standard additive character: 
\begin{equation}\label{eq:standard_character}
  \psi = \prod_{p\le \infty}\psi_p: \A/\Q\rightarrow \C^{\times},\; x=(x_p)\mapsto \psi(x)=e^{2\pi i (-x_\infty + \sum_{p<\infty}x_p')},
\end{equation}
where $x_p'\in \Q/\Z$ is the principal part of $x_p$. The subrepresentation of $\omega_f$ mentioned before acts on a finite-dimensional subspace of $S(V(\A_f))$, which is defined as follows: For each prime $p$ there is a $p$-adic lattice  $L_p=L\otimes \Z_p$ attached to $L$. Accordingly, we have the dual $p$-adic lattices $L_p'$. They are the $p$-part of $\widehat{L}=L\otimes \widehat{\Z}$ and $\widehat{L}'= L'\otimes \widehat{\Z}$, respectively.   $\widehat{L}'/\widehat{L}$ is canonically isomorphic to the discriminant group $L'/L$. For the rest of the paper we identify $\mu\in D$ with its image in $\widehat{L}'/\widehat{L}$. For $\mu\in D$  we define  $\varphi_\mu \in S(V(\A_f))$ with
\begin{align}\label{eq:familiy}
  \varphi_\mu = \mathbbm{1}_{\mu + \hat{L}}.
\end{align}
Here  $\mathbbm{1}_{\mu+\hat{L}}$ is the characteristic function of $\mu+\widehat{L}$.  
Then we consider the $|D|$-dimensional subspace
\begin{equation}\label{eq:space_char_func}
  S_L=\bigoplus_{\mu\in D}\C\varphi_\mu \subset S(V(\A_f)). 
\end{equation}
By means of $\frake_\mu\mapsto \varphi_\mu$ we can identify the spaces $\C[D]$ and $S_L$ with each other. Locally we have the isomorphism $L_p'/L_p\cong D_p$ of quadratic modules for all primes $p$ (see e.g. \cite{Ze}, Section 3). Therefore, we can write $D\cong \bigoplus_{p<\infty} L_p'/L_p$. On the level of the space $S_L$, this decomposition translates to the isomorphism
\begin{equation}\label{eq:local_decomp_S_L}
S_L\cong \bigotimes_{p< \infty} S_{L_p},\quad  \varphi_{\mu} \mapsto \bigotimes_{p<\infty}\varphi_p^{(\mu_p)},
\end{equation}
where $\mu= \sum_{p\mid |D|}\mu_p$ and $\varphi_p^{(\mu_p)} = \mathbbm{1}_{\mu_p+L_p}$. In particular, $\mu_p$ is trivial for all primes $p$ coprime to $|D|$. Accordingly, the $p$-part $S_{L_p}$ of $S_L$ is given by 
\begin{equation}\label{eq:local_S_L}
S_{L_p} =
\begin{cases}
  \bigoplus_{\mu\in L_p'/L_p}\C\varphi_p^{(\mu)}, & p\mid |D|,\\
  \C\varphi_p^{(0)}, & p\nmid |D|.
\end{cases}
\end{equation}
It is known that the space in \eqref{eq:space_char_func} is stable under the action of the group $\SL_2(\widehat{\Z})$ via the Weil representation $\omega_f$. 
Thus, $\SL_2(\Z_p)$ acts via the local Weil representation $\omega_p$ acts on $S_{L_p}$.

We then have 
 \[
\omega_f(\gamma_f)\varphi_\mu = \bigotimes_{p<\infty}\omega_p(\gamma_p)\varphi_p^{(\mu_p)}.
\]

Furthermore, under the identification of $S_L$ and $\C[D]$, $\omega_f$ coincides with the finite Weil representation $\rho_{D}$ in the following way 
  \begin{equation}\label{eq:adelic_weil_rerpr}
  \rho_D (\gamma) = \omega_f(\gamma_f),
  \end{equation}
  where $\gamma\in \Gamma$ and $\gamma_f\in \SL_2(\widehat{\Z})$ is the projection of $\gamma$ into $\SL_2(\widehat{\Z})$.  (see e.g.  \cite{YY}, p. 3456). 
  Based on this observation, we can define an extension of $\omega_f$ to $\calK$ in terms of the extension of $\rho_D$ to the group $Q(N)$. Indeed,
  for $N= \prod_{i=1}^rp_i^{\nu_{p_i}(N)}$ we have the usual projection homomorphisms
  \begin{align*}
    & \pi: \widehat{\Z}\rightarrow \prod_{p\mid N}\Z_p: \text{ the projection onto the places } p \mid N, \\
    & \pi_N: \prod_{p\mid N}\Z_p\rightarrow \Z/N\Z: \text{ the canonical projection of } \Z_p \text{ to } \Z/p^{\nu_p(N)}\Z \\
    &\text{ and the application of the Chinese remainder theorem subsequently.}
  \end{align*}
  We write  $\Pi$, $\Pi_N$ for the corresponding maps on matrices given by the componentwise application of $\pi$ and $\pi_N$. Then 
  \[
  \sigma_N = (\Pi_N\circ\Pi,\pi_N\circ\pi) 
  \]
  maps $\calK$ homomorphically to $Q(N)$. Thereby,
  \begin{equation}\label{eq:extension_omega_f}
    \omega_f(k) = \rho_D(\sigma_N(k))
  \end{equation}
  is well-defined. Taking \eqref{eq:adelic_weil_rerpr} into account, it  is obvious that \eqref{eq:extension_omega_f} indeed specifies an extension of $\omega_f$ to $\calK$.

  The same process yields an extension  of the local Weil representation $\omega_p$ to the group $\calK_p$ on the space $S_{L_p}$. Accordingly, we set
  \begin{equation}\label{eq:local_omega_p}
    \omega_p(k_p) = \begin{cases}
      \rho_{D_p}(\sigma_{p^{\nu_p(N)}}(k_p)),& p|N,\\
      \id_{S_{L_p}}, & p\nmid N.
    \end{cases}
  \end{equation}

The action of $\calK_0(N)$ on $S_L$ via $\omega_f$ plays a crucial role in relating scalar-valued and vector-valued adelic automorphic forms. The following result describes this action explicitly.
  
  \begin{lemma}\label{lem:weil_K_0_N}
    Let $\xi_D: \Q^\times\bs\A^\times\rightarrow \C^\times$ be the Hecke character induced by the Dirichlet character $\chi_D$ \eqref{eq:chi_D}.  
    \begin{enumerate}
      \item[i)]
 For $\widetilde{k}=(k, t)\in \calK_0(N)$ with $k=\kzxz{a}{b}{c}{d}$ we have 
  \begin{equation}\label{eq:weil_K_0_N}
    \omega_f(\widetilde{k})\varphi_\lambda =  \xi_D(a)e(bt^{-2}dq(\lambda))\varphi_{t^{-1}d\lambda}.
  \end{equation}
  In particular, for $\lambda = 0$ \eqref{eq:weil_K_0_N} specializes to
  \begin{equation}\label{eq:weil_K_0_N_zero}
    \omega_f(\widetilde{k})\varphi_0 = \xi_{D}(k)\varphi_0,
  \end{equation}
 which shows that the action of $\calK_0(N)$ on $\varphi_0$ is independent of the choice of the square root $t$  and descends to an action of $K_0(N)\cap G(\A)$. 
\item[ii)]
  For $\widetilde{z}_f=(z_f,t_f)\in \calZ(\A_f)\cap \calK_0(N)$ we have
  \begin{equation}\label{eq:weil_center}
    \omega_f(\widetilde{z}_f)\varphi_0 = \xi_D(z_f)\varphi_0.
  \end{equation}
  Therefore, the action of $\calZ(\A_f)\cap \calK_0(N)$ on $\varphi_0$ via $\omega_f$ is independent of the square root factor of $\widetilde{z}_f$.
  \end{enumerate}
    \end{lemma}
\begin{proof}
  The proof is a direct computation using the definition of $\omega_f$.
  
  $i)$: We write $\sigma_N(\widetilde{k}) = (\kzxz{a_N}{b_N}{c_N}{d_N}, t_N)$ and $\sigma_N(\kzxz{t}{0}{0}{t},t) = (\kzxz{t_N}{0}{0}{t_N},t_N)$. 
  By \eqref{eq:extension_omega_f} and the definition of  $\rho_{D}$ on $Q(N)$, we find
  \begin{align*}
    \omega((k,t))\varphi_\lambda = \rho_{D}\kzxz{t_N}{0}{0}{t_N}\rho_{D}\kzxz{t_N^{-1} a_N}{t_N^{-1} b_N}{t_N^{-1}c_N}{t_N^{-1}d_N}\varphi_\lambda,
  \end{align*}
  where $t_N^{-1}$ means the inverse of $t_N$ in $(\Z/N\Z)^\times$. 
  As $\widetilde{k}$ is assumed to be  in $\calK_0(N)$, the projection $\kzxz{t_N^{-1} a_N}{t_N^{-1} b_N}{t_N^{-1}c_N}{t_N^{-1}d_N}$ lifts to a matrix in $\Gamma_0(N)$. 
  Therefore,  combining Lemma \ref{lem:weil_rep_gamma_0_N}, the subsequent remark and \cite{BS}, (3.5) for the action of $\kzxz{t_N}{0}{0}{t_N}$, we obtain
  \begin{equation}\label{eq:weil_2}
    \begin{split}
      \omega(k,t)\varphi_\lambda & = \chi_{D}(t_N)\chi_{D}(t_N^{-1}a_N)e(-b_Nt_N^{-2}d_Nq(\lambda))\varphi_{t_N^{-1}d_N\lambda} \\
      &= \chi_{D}(a_N)e(-b_Nt_N^{-2}d_Nq(\lambda))\varphi_{t_N^{-1}d_N\lambda}.
    \end{split}
  \end{equation}
  By means of the definition of the induced Hecke character $\xi_D$ we can replace $\chi_D(a_N)$ with $\xi_D(a)$. Also, if we employ the same arguments leading to the well-known relation $\psi_f(r) = \psi_\infty(r)^{-1}$ for the additive character $\psi$, we may replace $e(-b_Nt_N^{-2}d_Nq(\lambda))$ with $e(bt^{-2}dq(\lambda))$. Similarly, by $\widehat{L'}/\widehat{L}\cong L'/L$ we may identify $\varphi_{t_N^{-1}d_N\lambda}$ with $\varphi_{t^{-1}d\lambda}$. This finally yields
  \begin{equation}\label{eq:weil_3}
    \begin{split}
      \omega(k,t)\varphi_\lambda & = \xi_D(a)e(bt^{-2}dq(\lambda))\varphi_{t^{-1}d\lambda}
    \end{split}
    \end{equation}

 This last formula specializes for $\lambda = 0$ to
  \begin{equation}\label{eq:local_weil_character}
    \omega_f(\widetilde{k})\varphi_0 = \xi_{D}(k)\varphi_0,
  \end{equation}
  where we used that $\xi_D(k)=\xi_D(d)$ and the fact that $\xi_D$ is quadratic.  

  $ii)$: Let $\widetilde{z}_f = (\kzxz{r}{0}{0}{r}, t)$ with $r^2 = t^2$ and $\sigma_N(\widetilde{z}_f) = (\kzxz{r_N}{0}{0}{r_N}, t_N)$ with $r_N^2 = t_N^2$. Then as in $i)$ we have
  \begin{equation}\label{eq:character_id}
    \begin{split}
    \omega_f(\kzxz{r}{0}{0}{r}, t)\varphi_0 &= \rho_D(\kzxz{r_N}{0}{0}{r_N}, t_N)\varphi_0\\
 &=\rho_D(\kzxz{t_N}{0}{0}{t_N},t_N)\rho_D\kzxz{r_Nt_N^{-1}}{0}{0}{r_Nt_N^{-1}}\varphi_0\\
    & =\frac{g(D)}{g_{t_N}(D)}\frac{g_{r_Nt_N^{-1}}(D)}{g(D)}\varphi_0,
    \end{split}
    \end{equation}
  where the last equations are based on \cite{BS}, (2.10) and (3.5). Then using $g_{t_N}(D) = g_{t_N^{-1}}(D)$ and \cite{BS}, (2.12) subsequently, we can replace the right-hand side of \eqref{eq:character_id} by
  \[
  \frac{g(D)}{g_{r_N}(D)}\varphi_0 = \xi_D(z_f)\varphi_0.
  \]
  which does not depend on the square root factor $t$. 
\end{proof}

\subsection{Vector-valued automorphic forms}
In this subsection, we introduce vector-valued modular forms for $\rho_D$ and their adelic counterparts and how they are related to each other by an analogue of the usual adelization map. 
For the definition of vector-valued modular forms we follow \cite{Br1}.

\begin{definition}\label{def:vec_val_mfs}
  Let $\kappa\in \Z$. A holomorphic function $f:\H\rightarrow \C[D]$ is called modular form of weight $\kappa$ with respect to $\rho_D$ and $\Gamma$ if
\[
f(\gamma\tau) = j(\gamma,\tau)^\kappa\rho_D(\gamma)f(\tau) \text{ for all } \gamma\in \Gamma,
\]
and $f$ is holomorphic at the cusp $\infty$.

We write $M_{\kappa}(\rho_D)$ for the space of all such functions and denote the subspace of cusp forms by $S_\kappa(\rho_D)$. 
\end{definition}

A vector-valued adelic automorphic form $F$ assigned to a cusp form $f\in S_\kappa(\rho_D)$ is a $S_L$-valued cuspidal function. With respect to the basis $\{\varphi_\mu\}_{\mu\in D}$ it can be written as $F = \sum_{\mu\in D}F_\mu\varphi_\mu$. In particular, it is an element of the Hilbert space $L^2(\calG(\Q)\setminus \calG(\A), \omega_f)$, which we define  in accordance with the classical scalar-valued $L^2$ space \eqref{eq:L2_scalar_valued}:
\begin{equation}\label{eq:L2_vector_valued}
  \begin{split}
      L^2(\calG(\Q)\setminus \calG(\A), \omega_f) = \left\{F:\calG(\A)\rightarrow S_L\;\left|
    \begin{array}{ll}
      \text{i)} &  F_\mu \text{ is measurable } \text{ for all } \mu \in D\\
      \text{ii)} & F(zg) = \omega_f(z_f)^{-1}F(g) \text{ for all }\\
      & z\in \calZ(\A)\\
      \text{iii)} & \int_{\overline{\calG}(\Q)\setminus \overline{\calG}(\A)}\langle F(g),F(g)\rangle dg <\infty
      \end{array}
    \right.\right\}
    \end{split}
\end{equation}
The main difference between the two $L^2$ spaces is that  the Weil representation $\omega_f$ replaces the  Hecke character $\xi$. 

The  subspace of cuspidal functions is defined entirely analogously to that in the scalar-valued case. Namely, 
\begin{equation}\label{eq:L2_cusp}
  \begin{split}
  L^2_0(\omega_f)
    &= \left\{F\in L^2(\calG(\Q)\setminus \calG(\A), \omega_f)\; \left|\; \int_{N_{\calG}(\Q)\setminus N_{\calG}(\A)}F_\mu(ng)dn = 0 \text{ for all } \mu \in D,   \text{  a. e. } g\in \calG(\A)\right.\right\}.
  \end{split}
  \end{equation}

We are now in a position to define an analogue of the adelization map in  \eqref{eq:adelization_map_scalar}:  
\begin{definition}\label{def:mod_adelic}
  Let $f\in S_\kappa(\rho_D)$ and $g\in \calG(\A)$ with $g=g_\Q(g_\infty \times k)$ according to Theorem \ref{thm:strong_approx_pairs}, where  $g_\Q\in \calG(\Q), g_\infty\in \GL_2(\R)^+$ and $k\in \calK$. With respect to this decomposition, we define the map $\mathscr{A}_{\omega}$ by
\begin{equation}\label{eq:mod_adelic}
f\mapsto \mathscr{A}_{\omega}(f)=F_f\;\text{ with } F_f(g)= \omega_f(k)^{-1}j(g_\infty,i)^{-\kappa}f(g_\infty i).
\end{equation}
  \end{definition}

It can be verified (see the Propositions 5.3 and 5.4 in \cite{St}) that $F_f$ is an element of $L^2_0(\omega_f)$. The following theorem (cf. \cite{St} Theorem 5.5) determines the image of $S_\kappa(\rho_L)$ under adelization map $\mathscr{A}_{\omega}$:

\begin{theorem}\label{thm:correspondence_mod-adelic}
  Let $A_{\kappa}(\omega_f)$ be the space of functions $F\in L^2_0(\omega_f)$ satisfying
  \begin{enumerate}
  \item[i)]
    $    F(gk) = \omega_f(k)^{-1}F(g)$ for all $k\in \calK$ and all $g\in\calG(\A)$
  \item[ii)]
    $F(g\kzxz{\cos(\theta)}{\sin(\theta)}{-\sin(\theta)}{\cos(\theta)}) = e^{i\kappa\theta}F(g)$ for all $\theta\in [0,2\pi)$ and all $g\in \calG(\A)$
  \item[iii)]
    All the components $F_\mu$ of  $F$, considered as a function of $\GL_2(\R)^+$ alone, satisfy the differential equation $L F_\mu = 0$.  Here $L$ is the differential operator given by
    \begin{equation}
    L = e^{-2i\theta}\left(-2iy\frac{\partial}{\partial x}+2y\frac{\partial}{\partial y} +i \frac{\partial}{\partial \theta}\right)
    \end{equation}
    with respect to the coordinates of the Iwasawa decomposition 
of $g_\infty\in \GL_2(\R)^+$. 
  \end{enumerate}
  Then the map $\mathscr{A}_\omega$ defines an isometry from $S_\kappa(\rho_L)$ onto $A_{\kappa}(\omega_f)$. 
\end{theorem}

\section{Liftings of adelic automorphic forms}\label{sec:lifting_automorphic_forms}
The purpose of this section is to establish adelic analogues of the lifting maps between scalar-valued and vector-valued modular forms, which were studied in detail by Scheithauer in \cite{Sch1} and others (see e.g. \cite{BB}, \cite{Zh1} or \cite{Zh2}). We construct the corresponding lifting maps between scalar-valued and vector-valued automorphic forms and show that they are compatible with their classical counterparts under adelization.

\begin{theorem}[\cite{Sch1}, Theorem 3.1]
  Let $D$ be a discriminant form 
 and $f\in M_\kappa(\Gamma_0(N),\chi_D)$.
  Then 
  \begin{equation}\label{eq:lift_mf}
    \calL_{\rho_D}(f) = \sum_{\gamma\in \Gamma_0(N)\bs \Gamma}(f\mid_\kappa \gamma)\rho_D(\gamma)^{-1}\frake_0
  \end{equation}
  is a vector-valued modular form for $\rho_D$ and weight $\kappa$, which is invariant under the automorphisms of the discriminant form. 
\end{theorem}

On the other hand, using Prop. 4.5 in \cite{Sch2} and the transformation behaviour of a modular form $f=\sum_{\lambda\in D}f_\lambda\frake_\lambda\in M_\kappa(\rho_D)$, we find that $f_0 = \langle f, \frake_0\rangle$ is in $M_\kappa(\Gamma_0(N),\chi_D)$.

  Similar relations hold between the automorphic forms associated to both types of modular forms. The following theorem shows that the zero component $F_0$ of a vector-valued automorphic form is an element of $A_\kappa(K_0(N)\cap G(\A), \xi_{D_{|_{G(\A)}}})$. 

  \begin{theorem}\label{thm:lift_scalar_vec_valued}
    The mapping $\Phi_D: A_\kappa(\omega_f)\rightarrow A_\kappa(K_0(N)\cap G(\A), \xi_{D_{|_{G(\A)}}})$,
    \begin{equation}\label{eq:scalar_valued_adelic_mf}
      F\mapsto \Phi_D(F)=\langle F, \varphi_0\rangle
    \end{equation}
    assigns to a vector-valued automorphic form $F$ a scalar-valued adelic automorphic form of weight $\kappa$ and level $K_0(N)\cap G(\A)$ on $G(\A)$  with central character $\xi_D = \omega^{-1}_{f_{|_{\calK_0(N)}}}$ on $K_0(N)$.
  \end{theorem}
  \begin{proof}
    By  Lemma \ref{lem:weil_K_0_N}, $i)$, $\xi_D$ is induced from the Dirichlet character $\chi_D$. Clearly, $\langle F,\varphi_0\rangle$ is a priori a function on $\calG(\Q)\bs \calG(\A)$. However, since $\omega_f$ is unitary with respect to $\langle \cdot, \cdot \rangle$, Lemma \ref{lem:weil_K_0_N} implies 
    \begin{align*}
      \langle F(g(1,u)), \varphi_0\rangle &= \langle F(g), \omega_f(1,u)\varphi_0\rangle \\
      & = \langle F(g), \varphi_0\rangle. 
    \end{align*} 
Thus, $\Phi_D$ is independent of the choice of the square root and thereby descends to a function on $G(\Q)\bs G(\A)$. 
    Next, we  verify the correct transformation behaviour with respect to $\calK_0(N)$. Let $\widetilde{k} = (k,t)\in \calK_0(N)$. Then using Lemma \ref{lem:weil_K_0_N} yields
   \begin{align*}
      \Phi_D(g\widetilde{k}) &= \langle F(g\widetilde{k}),\varphi_0\rangle = \langle \omega_f^{-1}(\widetilde{k})F(g),\varphi_0\rangle\\
      &= \langle F(g),\omega_f(\widetilde{k})\varphi_0\rangle  = \xi_D(k)\Phi_D(g).
   \end{align*}
    We further have
    \[
    \int_{\overline{\calG}(\Q)\bs \overline{\calG}(\A)}|\Phi_D(F)(g)|^2dg \le \int_{\overline{\calG}(\Q)\bs \overline{\calG}(\A)}\langle F(g), F(g)\rangle dg < \infty.
    \]
    Also, let $\widetilde{z} = (z,t)\in \calZ(\A)$. Then by Lemma \ref{lem:weil_K_0_N}, $ii)$,
    \begin{align*}
      \Phi_D(F)(zg) &= \langle F(zg),\varphi_0\rangle = \langle \omega_f(z_f)^{-1} F(g),\varphi_0\rangle\\
      & = \langle F(g), \omega_f(z_f)\varphi_0\rangle  = \xi_D(z)\Phi_D(g),
    \end{align*}
    where the last equality uses the fact that $\xi_D$ is trivial on $G(\Q)$ and $\GL_2(\R)^+$. Combining the last two established properties of $\Phi_D$, confirms that $\Phi_D\in L^2(G(\Q)\bs G(\A), \xi_D)$.

The remaining defining properties, including cuspidality, are inherited componentwise from $F$.    
  \end{proof}

 Similarly, the lifting \eqref{eq:lift_mf} admits an adelic counterpart relating scalar-valued and vector-valued automorphic forms.

  \begin{theorem}\label{thm:lift_scalar_vector_valued}
    Let $F\in A_\kappa(K_0(N)\cap G(\A), \xi_{D_{|_{G(\A)}}})$ with $\xi_D$ defined as in Theorem \ref{thm:lift_scalar_vec_valued} and
    \[
    p: \calG(\A)\rightarrow G(\A), \quad p(g,t)\mapsto g
    \]
    the projection to the first component. 
    Then
    \begin{equation}\label{eq:lift_scalar_vector_valued}
      g\mapsto \calL_\omega(F)(g)=\sum_{k\in \calK/\calK_0(N)}F(p(gk))\omega_f(k)\varphi_0
    \end{equation}
    is an adelic vector-valued automorphic form in $A_\kappa(\omega_f)$.  
  \end{theorem}
  \begin{proof}
    The first step is to verify that $\calL_\omega(F)$ is well-defined. Passing from a coset representative $k$ to $kk_0$ with $k_0\in \calK_0(N)$ gives
    \begin{align*}
      & \sum_{k\in \calK/\calK_0(N)}F(p(gkk_0))\omega_f(kk_0)\varphi_0 = \sum_{k\in \calK/\calK_0(N)}F(p(gk))\xi_D(p(k_0))\omega_f(k)\omega_f(k_0)\varphi_0\\
    \end{align*}
   By Lemma \ref{lem:weil_K_0_N}, $\omega_f(k_0)\varphi_0$ is equal to $\xi_D(p(k_0))$. Since $\xi_D$ is quadratic, the factor $\xi_D(p(k_0))^2$  cancels out. 
    A priori, $\calL_\omega(F)$ is a function on $\calG(\Q)\bs \calG(\A)$. In fact, it is even an element in the space $L^2(\calG(\Q)\bs \calG(\A),\omega_f)$:  The first property follows immediately from the corresponding one of $F$. As for the transformation behaviour with respect to $\calZ(\A)$, we exploit once again Lemma \ref{lem:weil_K_0_N}, $ii)$. Combined with the corresponding behaviour of $F$, we obtain
    \begin{align*}
      \calL_\omega(zg) & = \sum_{k\in \calK/\calK_0(N)}F(p(z)p(gk))\omega_f(k)\varphi_0\\
      & = \sum_{k\in \calK/\calK_0(N)}F(p(gk))\omega_f(k)\omega_f(p(z_f))\varphi_0 = \omega_f(z_f)^{-1}\calL_\omega(F)(g).
      \end{align*}
    For the third property we observe that
    \begin{align*}
      \int_{\overline{\calG}(\Q)\setminus \overline{\calG}(\A)}\langle \calL_\omega(F)(g),\calL_\omega(F)(g)\rangle dg &  = \sum_{k,k'\in \calK/\calK_0(N)}\langle \omega_f(k)\varphi_0,\omega_f(k')\varphi_0\rangle\int_{\overline{\calG}(\Q)\setminus \overline{\calG}(\A)}F(p(gk))\overline{F(p(gk'))}dg \\
      &= \sum_{k,k'\in \calK/\calK_0(N)}\langle \omega_f(k'^{-1}k)\varphi_0,\varphi_0\rangle \int_{\overline{\calG}(\Q)\setminus \overline{\calG}(\A)}F(p(gk'^{-1}k))\overline{F(p(g))}dg,
      \end{align*}
    where the last equality is obtained by the change of variables $g\mapsto gk'^{-1}$. Since $\calK/\calK_0(N)$ is finite, it suffices to confirm that the integral
    \[
   \int_{\overline{\calG}(\Q)\setminus \overline{\calG}(\A)}F(p(gk'^{-1}k))\overline{F(p(g))}dg 
   \]
   is finite for any pair $k, k'$. The integrand does not depend on the second component of $\calG(\Q)\bs\calG(\A)$ and therefore the latter integral is equal to 
   \begin{equation}\label{eq:integral}
   \int_{\overline{G}(\Q)\bs \overline{G}(\A)}F(p(gk'^{-1}k))\overline{F(p(g))}dp(g)
   \end{equation}
by \eqref{eq:quotient_center}. 
     The integral in \eqref{eq:integral}  can be written as the scalar product $(\cdot,\cdot)$ on the space $L^2(G(\Q)\bs G(\A),\xi_D)$ of $R(k'^{-1}k)F$ and $F$, where $R(k'^{-1}k)F(g)=F(gk'^{-1}k)$. The Cauchy-Schwarz inequality then gives
   \[
   |(R(k'^{-1})F,F)|\le \|R(k'^{-1})F\|_2\|F\|_2 = \|F\|_2^2 < \infty.
   \]
   Cuspidality of $\calL_\omega(F)$ follows because the maximal unipotent subgroup of $\calG(\A)$ identifies with $N(\A)\times \{1\}$ by \eqref{eq:unipotent_group}. Thus,
   \begin{align*}
     \int_{N_{\calG}\bs N_{\calG}(\A)}\calL_\omega(F)_\mu(ng)dn 
     &= \sum_{k\in \calK/\calK_0(N)}\langle\omega_f(k)\varphi_0,\varphi_\mu\rangle \int_{N(\Q)\bs N(\A)} F(p(n)p(gk))dn.
   \end{align*}

  It remains to prove that the properties of Theorem \ref{thm:correspondence_mod-adelic} are satisfied.
   The correct transformation behaviour with respect to right multiplication of $\calK$ can be seen as follows: For any $k'\in \calK$ we have
   \begin{align*}
     \calL_\omega(F)(gk') &= \sum_{k\in \calK/\calK_0(N)}F(p(gk'k))\omega_f(k)\varphi_0\\
     &= \sum_{k\in \calK/\calK_0(N)}F(p(gk'k))\omega_f(k')^{-1}\omega_f(k'k)\varphi_0\\
     &=\omega_f(k')^{-1}\calL_\omega(F)(g).
   \end{align*}
  The last equation holds because $k'k$ runs through a complete set of representatives of $\calK/\calK_0(N)$ if $k$ does. 
   The transformation behaviour of $\calL(F)$ with respect to $\SO(2)$ is an easy consequence of the fact that elements in $\calK$ and $\SO(2)$ commute in $\calG(\A)$ and the corresponding property of $F$.
   The property $iii)$ of Theorem \ref{thm:correspondence_mod-adelic} follows immediately from the corresponding property of $F$. 
  \end{proof}

  Finally, we show that the lifts $\calL_{\rho_D}$ and $\calL_{\omega}$ commute with the adelization maps on both sides. To that end, let us first remark that although $\mathscr{A}_{\xi_D}(f)$ is an automorphic form on $\GL_2(\Q)\bs \GL_2(\A)$, we work with its restriction to $G(\Q)\bs G(\A)$ since $\calL_\omega$ is defined on the latter.  By Theorem \ref{rem:relation_automorphic_forms}, no information is lost by passing to the restriction.

  \begin{theorem}\label{thm:commuation_adelization_lifts}
    The diagram
 
    \begin{equation}\label{diag:commutative_diagram}
    \begin{tikzcd}
  A_\kappa(K_0(N)\cap G(\A), \xi_{D_{|_{G(\A)}}})     \arrow[r, "\calL_\omega"] 
      &  A_\kappa(\omega_f) \\
     S_{\kappa}(\Gamma_0(N),\chi_D)  \arrow[r,"\calL_{\rho_D}" ] \arrow[u,  "\mathscr{A}_{\xi_D}"] 
      &  S_\kappa(\rho_D)  \arrow[u, "\mathscr{A}_\omega"]
\end{tikzcd}
    \end{equation}
    
    is commutative. 
  \end{theorem}
  \begin{proof}
    We have to show that for any $f\in S_\kappa(\Gamma_0(N),\chi_D)$ the identity
    \begin{equation}\label{eq:commuation_rel}
      (\mathscr{A}_\omega\circ \calL_{\rho_D})(f) = (\calL_\omega\circ \mathscr{A}_{\xi_D})(f)
    \end{equation}
    holds.

    We compute both sides of \eqref{eq:commuation_rel} and compare the resulting expressions.
    
 Using \eqref{eq:lift_mf} and \eqref{eq:mod_adelic},  we find for the left-hand side of \eqref{eq:commuation_rel} for any $g=g_\Q(g_\infty \times k')\in \calG(\A)$, 
    \begin{equation}\label{eq:calc_left_hand_side}
      \begin{split}
      (\mathscr{A}_\omega\circ \calL_{\rho_D})(f)(g_\Q(g_\infty\times k')) &= \omega_f(k')^{-1}j(g_\infty,i)^{-\kappa}\calL_{\rho_D}(f)(g_\infty i)\\
        &= \omega_f(k')^{-1}j(g_\infty,i)^{-\kappa}\sum_{\gamma\in \Gamma_0(N)\bs \Gamma}j(\gamma,g_\infty i)^{-\kappa}f(\gamma g_\infty i)\rho_D(\gamma)^{-1}\varphi_0\\
        & \omega_f(k')^{-1}\sum_{\gamma\in \Gamma_0(N)\bs \Gamma}j(\gamma g_\infty,i)^{-\kappa}f(\gamma g_\infty i)\rho_D(\gamma)^{-1}\varphi_0.
        \end{split}
      \end{equation}
 On the other hand, by \eqref{eq:adelization_map_scalar} and \eqref{eq:lift_scalar_vector_valued}
 \begin{equation}\label{eq:calc_right_hand_side}
   \begin{split}
      (\calL_\omega\circ \mathscr{A}_{\xi_D})(f)(g_\Q(g_\infty \times k')) & = \sum_{k\in \calK/\calK_0(N)}F_f(p(g_\Q(g_\infty \times k')(1_\infty\times k)))\omega_f(k)\varphi_0\\
     &= \omega_f(k')^{-1}\sum_{k\in \calK/\calK_0(N)}F_f(p(g_\Q(g_\infty\times k'k)))\omega_f(k'k)\varphi_0.
     \end{split}
    \end{equation}
Since $[\calK:\calK_0(N)] = [\Gamma:\Gamma_0(N)]$ and $\calK_0(N)\cap \calG(\Q)= \Gamma_0(N)$ (see e.g. \cite{KL}, Sect. 7.8 and Lemma 13.2, which also holds for the groups in this paper), we may choose a system of representatives of $\calK/\calK_0(N)$ consisting of elements of $\Gamma$. 
Taking this into account, the right-hand side of \eqref{eq:calc_right_hand_side} can be written as
    \begin{align*}
      & \omega_f(k')^{-1}\sum_{\gamma\in \Gamma/\Gamma_0(N)}F_f(p(g_\Q)(g_\infty \times \gamma))\omega_f(\gamma)\varphi_0\\
      &= \omega_f(k')^{-1}\sum_{\gamma\in \Gamma/\Gamma_0(N)}F_f(p(g_\Q)\gamma(\gamma^{-1}g_\infty \times 1_f)\omega_f(\gamma)\varphi_0 \\
      &= \omega_f(k')^{-1}\sum_{\gamma\in \Gamma/\Gamma_0(N)}j(\gamma^{-1}g_\infty,i)^{-\kappa}f(\gamma^{-1}g_\infty i)\omega_f(\gamma)\varphi_0\\
      &= \omega_f(k')^{-1}\sum_{\gamma'\in \Gamma_0(N)\bs \Gamma}j(\gamma' g_\infty,i)^{-\kappa}f(\gamma' g_\infty i)\omega_f(\gamma'^{-1})\varphi_0.
    \end{align*}
    Comparing this last expression with the right-hand side of \eqref{eq:calc_left_hand_side}, yields the identity \eqref{eq:commuation_rel}. 
    \end{proof}

\end{document}